\documentclass[11pt,a4paper]{amsart}
\usepackage{graphicx,multirow,array,amsmath,amssymb}

\newtheorem{theorem}{Theorem}

\newtheorem{proposition}[theorem]{Proposition}

\begin{document}

\title{A new intrinsically knotted graph with 22 edges}
\author[H. Kim]{Hyoungjun Kim}
\address{Department of Mathematics, Korea University, Seoul 02841, Korea}
\email{kimhjun@korea.ac.kr}
\author[H. J. Lee]{Hwa Jeong Lee}
\address{Department of Mathematical Sciences, KAIST, 291 Daehak-ro, Yuseong-gu, Daejeon 305-701, Korea}
\email{hjwith@kaist.ac.kr}
\author[M. Lee]{Minjung Lee}
\address{Department of Mathematics, Korea University, Seoul 02841, Korea}
\email{mjmlj@korea.ac.kr}
\author[T. Mattman]{Thomas Mattman}
\address{Department of Mathematics and Statistics, California State University, Chico, Chico CA 95929-0525, USA}
\email{TMattman@CSUChico.edu}
\author[S. Oh]{Seungsang Oh}
\address{Department of Mathematics, Korea University, Seoul 02841, Korea}
\email{seungsang@korea.ac.kr}

\thanks{2010 Mathematics Subject Classification: 57M25, 57M27, 05C10}
\thanks{The corresponding author(Seungsang Oh) was supported by the National Research Foundation of Korea(NRF) grant funded by the Korea government(MSIP) (No. NRF-2014R1A2A1A11050999).}
\thanks{This work was supported by the National Research Foundation of Korea(NRF) grant
funded by the Korea government(MEST) (No. 2011-0027989).}

\begin{abstract}
A graph is called intrinsically knotted if every embedding of the graph contains a knotted cycle.
Johnson, Kidwell, and Michael showed that intrinsically knotted graphs have at least 21 edges.
Recently Lee, Kim, Lee and Oh (and, independently, Barsotti and Mattman) 
proved there are exactly 14 intrinsically knotted graphs with 21 edges 
by showing that $H_{12}$ and $C_{14}$ are the only triangle-free intrinsically knotted graphs of size 21.
Our current goal is to find the complete set of intrinsically knotted graphs with 22 edges.
To this end, using the main argument in~\cite{LKLO},
we seek triangle-free intrinsically knotted graphs.
In this paper we present a new intrinsically knotted graph with 22 edges, called $M_{11}$.
We also show that there are exactly three triangle-free intrinsically knotted graphs of size 22
among graphs having at least two vertices with degree 5: 
cousins 94 and 110 of the $E_9+e$ family, and $M_{11}$.
Furthermore, there is no triangle-free intrinsically knotted graph with 22 edges
that has a vertex with degree larger than 5.
\end{abstract}

\maketitle

\section{Introduction} \label{sec:intro}

Throughout the paper we will take an embedded graph to mean a graph embedded in $R^3$.
We call a graph $G$ {\em intrinsically knotted\/}
if every embedding of the graph contains a non-trivially knotted cycle.
Conway and Gordon \cite{CG} showed that $K_7$, the complete graph with seven vertices,
is an intrinsically knotted graph.
Foisy \cite{F} showed that $K_{3,3,1,1}$ is also intrinsically knotted.
A graph $H$ is a {\em minor\/} of another graph $G$
if it can be obtained from $G$ by contracting or deleting some edges.
If a graph $G$ is intrinsically knotted and has no proper minor
that is intrinsically knotted, $G$ is said to be {\em minor minimal intrinsically knotted\/}.
Robertson and Seymour \cite{RS} proved that
there are only finitely many minor minimal intrinsically knotted graphs,
but finding the complete set is still an open problem.
A $\nabla Y$ {\em move\/} is an exchange operation on a graph
that removes all edges of a triangle $abc$ and
then adds a new vertex $v$ and three new edges $va, vb$ and $vc$.
Its reverse operation is called a $Y \nabla$ {\em move\/} as follows:

\begin{figure}[h]
\includegraphics{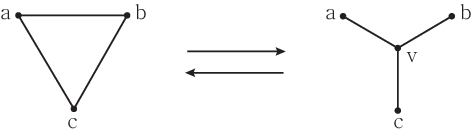}
\end{figure}

Since a $\nabla Y$ move preserves intrinsic knottedness \cite{MRS},
we will concentrate on triangle-free graphs in this paper.
It is known \cite{{CG},{F},{KS},{MRS}} that
$K_7$ and the 13 graphs obtained from $K_7$ by $\nabla Y$ moves,
and $K_{3,3,1,1}$ and the 25 graphs obtained from $K_{3,3,1,1}$ by $\nabla Y$ moves
are minor minimal intrinsically knotted.

Johnson, Kidwell and Michael \cite{JKM} showed that intrinsically knotted graphs have at least 21 edges.
Recently two groups, working independently,
showed that $K_7$ and the 13 graphs obtained from $K_7$ by $\nabla Y$ moves are the only intrinsically
knotted graphs with 21 edges.
This gives us the complete set of 14 minor minimal intrinsically knotted graphs with 21 edges.
Lee, Kim, Lee and Oh \cite{LKLO} proceeded by showing that the only triangle-free
intrinsically knotted graphs with 21 edges are $H_{12}$ and $C_{14}$.
Barsotti and Mattman \cite{BM} relied on connections with 2-apex graphs.

We say that a graph is {\em intrinsically knotted or completely 3-linked}
if every embedding of the graph into the 3-sphere contains
a nontrivial knot or a 3-component link each of whose 2-component sublinks is nonsplittable.
Hanaki, Nikkuni, Taniyama and Yamazaki \cite{HNTY} constructed 20 graphs
derived from $H_{12}$ and $C_{14}$ by $Y \nabla$ moves as in Figure \ref{fig1},
and they showed that these graphs are minor minimal intrinsically knotted or completely 3-linked graphs.
Furthermore they showed that the six graphs $N_9$, $N_{10}$, $N_{11}$, $N'_{10}$, $N'_{11}$ and $N'_{12}$
are not intrinsically knotted.
This has also been shown by Goldberg, Mattman and Naimi independently \cite{GMN}.

\begin{figure}[h]
\includegraphics{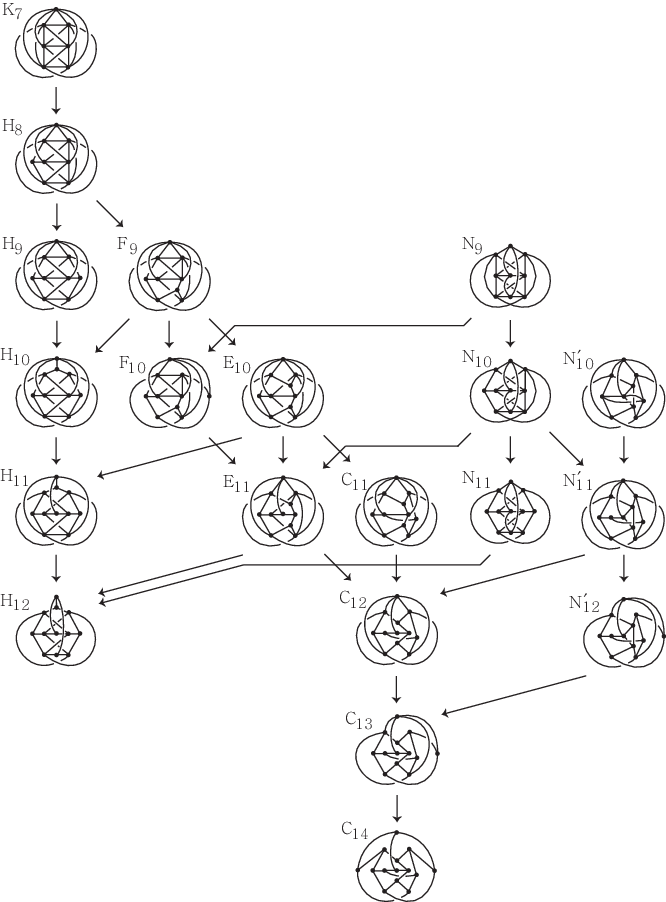}
\caption{$K_7$ and $19$ more relative graphs}
\label{fig1}
\end{figure}

We say two graphs $G$ and $G^{\prime}$ are {\em cousins} of each other
if $G^{\prime}$ is obtained from $G$ by a finite sequence of $\nabla Y$ and $Y \nabla$ moves.
The set of all cousins of $G$ is called the $G$ {\em family}.
The $K_{3,3,1,1}$ family consists of 58 graphs, of which
26 graphs were previously known to be minor minimal intrinsically knotted.
In \cite{GMN} they showed that the remaining 32 graphs are also minor minimal intrinsically knotted.
As in Figure \ref{fig2} the graph $E_9+e$ is obtained from $N_9$ by adding the new edge $e$.
The $E_9+e$ family consists of 110 graphs.
They also showed that all of these graphs are intrinsically knotted,
and exactly 33 of them are minor minimal intrinsically knotted.

\begin{figure}[h]
\includegraphics{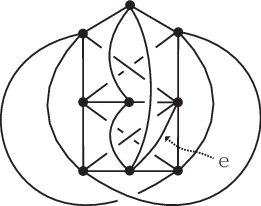}
\caption{$E_9+e$}
\label{fig2}
\end{figure}

By combining the $K_{3,3,1,1}$ family and the $E_9+e$ family,
all 168 graphs were already known to be intrinsically knotted graphs with 22 edges,
and 14 graphs among them are triangle-free.
These triangle-free intrinsically knotted 14 graphs are
the four cousins 29, 31, 42 and 53 in the $K_{3,3,1,1}$ family
and the ten cousins 43, 56, 87, 89, 94, 97, 99, 105, 109 and 110 in the $E_9+e$ family.
Especially, only the two cousins, 94 and 110, of the $E_9+e$ family drawn in Figure \ref{fig3}
have at least two vertices with degree 5 or more.
Indeed if we contract the edges $l_1$ and $l_2$ of the cousins 94 and 110,
we get the minor minimal intrinsically knotted graphs $H_{11}$ and $H_9$
of the $K_7$ family, respectively.

\begin{figure}[h]
\includegraphics{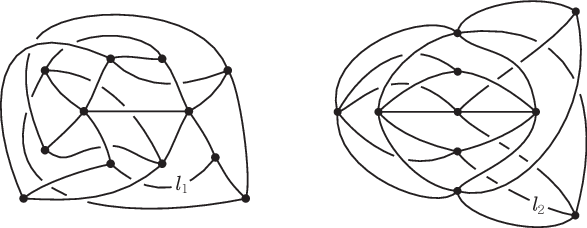}
\caption{Two cousins 94 and 110 of the $E_9+e$ family}
\label{fig3}
\end{figure}

We introduce a previously unknown graph $M_{11}$, shown in Figure \ref{fig4},
that is a triangle-free intrinsically knotted (but not minor minimal) graph
with 22 edges having four vertices with degree 5.
Note that $M_{11}$ does not appear in the $K_{3,3,1,1}$ family or the $E_9+e$ family.
Indeed if we contract the edge $l_3$,
we get the minor minimal intrinsically knotted graph $H_{10}$ of the $K_7$ family.

\begin{figure}[h]
\includegraphics{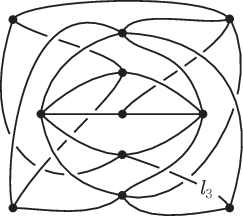}
\caption{New intrinsically knotted graph $M_{11}$}
\label{fig4}
\end{figure}

The goal of this research is to find the complete set of intrinsically knotted graphs with 22 edges.
To this end, using the main argument in~\cite{LKLO},
we seek triangle-free intrinsically knotted graphs with 22 edges.
In this paper, we find all such graphs among those having
at least two vertices with degree 5 or a vertex with degree larger than 5.

\begin{theorem}\label{thm:main}
There are exactly three triangle-free intrinsically knotted graphs with 22 edges among graphs
having at least two vertices with degree 5, the two cousins 94 and 110 of the $E_9+e$ family,
and a previously unknown graph named $M_{11}$ (Figure~\ref{fig3}~and~\ref{fig4}).

Furthermore, there is no triangle-free intrinsically knotted graph with 22 edges
that has a vertex with degree larger than 5.
\end{theorem}

The remaining sections of the paper are devoted to the proof of Theorem \ref{thm:main}.
This paper relies on the technical machinery developed in \cite{LKLO}.
In Section \ref{sec:term}, we review this material.
Henceforth, $a$ denotes a vertex of maximal degree in $G$.
The proof is divided into three parts according to the degree of $a$.
In Section \ref{sec:deg6}, we show that any graph $G$ with $\deg(a) \geq 6$
cannot be triangle-free intrinsically knotted.
In Section \ref{sec:deg5dist2}, we show that the only triangle-free intrinsically knotted graphs
with two vertices $a$ and $b$ of degree 5 with distance at least 2 are
cousin 110 of the $E_9+e$ family and the graph $M_{11}$.
In Section \ref{sec:deg5dist1}, we show that the only triangle-free intrinsically knotted graph
with two vertices $a$ and $b$ of degree 5 with distance 1
is cousin 94 of the $E_9+e$ family.

\section{Terminology} \label{sec:term}

Henceforth, let $G=(V,E)$ denote a triangle-free graph with 22 edges.
Here $V$ and $E$ denote the sets of vertices and edges of $G$ respectively.
For any two distinct vertices $a$ and $b$,
let $\widehat{G}_{a,b}=(\widehat{V}_{a,b}, \widehat{E}_{a,b})$ denote the graph
obtained from $G$ by deleting two vertices $a$ and $b$,
and then contracting edges adjacent to vertices of degree 1 or 2,
one by one repeatedly, until no vertices of degree 1 or 2 remain.
Removing vertices means deleting interiors of all edges adjacent to these vertices
and remaining isolated vertices.
The distance, denoted by dist$(a,b)$, between $a$ and $b$ is the number of edges
in the shortest path connecting them.
The degree of $a$ is denoted by $\deg(a)$.
To count the number of edges of $\widehat{G}_{a,b}$, we have some notation.

\begin{itemize}
\item $E(a)$ is the set of edges that are adjacent to $a$.
\item $E(V') = \cup_{a \in V'} E(a)$ for a subset $V'$ of $V$.
\item $V(a)=\{c \in V\ |\ \mbox{dist}(a,c)=1\}$
\item $V_n(a)=\{c \in V\ |\ \mbox{dist}(a,c)=1,\ \deg(c)=n,\}$
\item $V_n(a,b)=V_n(a) \cap V_n(b)$
\item $V_Y(a,b)=\{c \in V\ |\ \exists \ d \in V_3(a,b) \ \mbox{such that}
\ c \in V_3(d) \setminus \{a,b\}\}$
\end{itemize}

Obviously in $G \setminus \{a,b\}$ for some distinct vertices $a$ and $b$,
each vertex of $V_3(a,b)$ has degree 1.
Also each vertex of $V_3(a), V_3(b)$ (but not of $V_3(a,b)$) and $V_4(a,b)$ has degree 2.
Therefore to derive $\widehat{G}_{a,b}$ all edges adjacent to $a,b$ and $V_3(a,b)$
are deleted from $G$,
followed by contracting one of the remaining two edges adjacent to each vertex of
$V_3(a)$, $V_3(b)$, $V_4(a,b)$ and $V_Y(a,b)$ as in Figure \ref{fig5}(a).
Thus we have the following equation counting the number of edges of $\widehat{G}_{a,b}$
which is called the {\em count equation\/};
$$|\widehat{E}_{a,b}| = 22 - |E(a)\cup E(b)| -
(|V_3(a)|+|V_3(b)|-|V_3(a,b)|+|V_4(a,b)|+|V_Y(a,b)|)$$

For short, write $NE(a,b) = |E(a)\cup E(b)|$ and $NV_3(a,b) = |V_3(a)|+|V_3(b)|-|V_3(a,b)|$.
Provided that $a$ and $b$ are adjacent vertices (i.e. dist$(a,b)=1$),
$V_3(a,b), V_4(a,b)$ and $V_Y(a,b)$ are all empty sets because $G$ is triangle-free.
Note that the derivation of $\widehat{G}_{a,b}$ must be handled slightly differently
when there is a vertex $c$ in $V$ such that more than one vertex of $V(c)$ is contained
in $V_3(a,b)$ as in Figure \ref{fig5}(b).
In this case we usually delete or contract more edges even though $c$ is not in $V_Y(a,b)$.

\begin{figure}[h]
\includegraphics{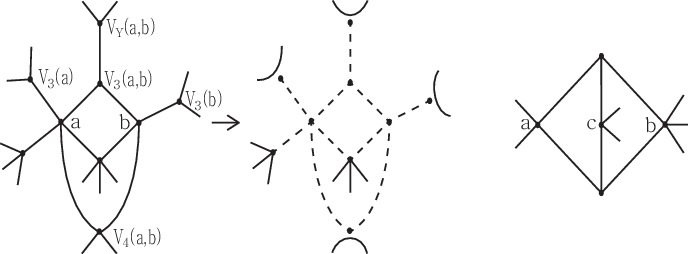}\\
\vspace{5mm}
\hspace{22mm}(a)\hspace{58mm} (b)
\caption{Deriving $\widehat{G}_{a,b}$}
\label{fig5}
\end{figure}

The following proposition, which was mentioned in \cite{LKLO},
gives an important condition for a graph not to be intrinsically knotted.
Note that $K_{3,3}$ is a triangle-free graph and every vertex has degree 3.

\begin{proposition} \label{prop:planar}
If $\widehat{G}_{a,b}$ satisfies one of the following two conditions,
then $G$ is not intrinsically knotted.
\begin{itemize}
\item[(1)] $|\widehat{E}_{a,b}| \leq 8$, or
\item[(2)] $|\widehat{E}_{a,b}|=9$ and $\widehat{G}_{a,b}$ is not homeomorphic to $K_{3,3}$.
\end{itemize}
\end{proposition}

\begin{proof}
If $|\widehat{E}_{a,b}| \leq 8$, then $\widehat{G}_{a,b}$ is a planar graph.
Also if $|\widehat{E}_{a,b}|=9$, then $\widehat{G}_{a,b}$ is
either a planar graph or homeomorphic to $K_{3,3}$.
It is known that if $\widehat{G}_{a,b}$ is planar,
then $G$ is not intrinsically knotted \cite{{BBFFHL},{OT}}.
\end{proof}

Since the intrinsically knotted graphs with 21 edges are known,
it is sufficient to consider simple and connected graphs having no vertex of degree 1 or 2.
Our process is to construct all possible such triangle-free graphs $G$ with 22 edges,
delete two suitable vertices $a$ and $b$ of $G$,
and then count the number of edges of $\widehat{G}_{a,b}$.
If $\widehat{G}_{a,b}$ has 9 edges or less and is not homeomorphic to $K_{3,3}$,
then we can use Proposition \ref{prop:planar} to show that $G$ is not intrinsically knotted.

We introduce more notation.
Let $H$ be a connected subgraph of $G$.
Let $\overline{E}(H)$ denote the set of all edges of $G \setminus H$,
and $\overline{V}(H)$ denote the set of all vertices of $G$
that are not a vertex of $H$ with degree more than 1.

\begin{itemize}
\item $\overline{E}(H) = \{\overline{e}_1, \cdots, \overline{e}_m\}$ \ \ for some $m$
\item $\overline{V}(H) = \{\overline{v}_1, \cdots, \overline{v}_n\}$ \ \ for some $n$ \
     with $\deg(\overline{v}_i) \geq \deg(\overline{v}_{i+1})$
\item $[\overline{V}(H)] = [\deg(\overline{v}_1), \cdots, \deg(\overline{v}_n)]$
\item $|[\overline{V}(H)]| = \deg(\overline{v}_1)+ \cdots + \deg(\overline{v}_n)$
\end{itemize}
We will use $\overline{e}_i$ and $\overline{v}_j$ without referring to $H$,
and sometimes we call $\overline{e}_i$ an {\em extra edge\/}.
To visualize $G$, we basically proceed through the following steps.
First choose a proper subgraph $H$ of $G$.
Draw $\overline{E}(H)$ apart from $H$ as in Figure \ref{fig6}(a).
$|[\overline{V}(H)]|$ indicates the number of vertices of degree 1 of $H$ and $\overline{E}(H)$,
and they are merged into vertices of $\overline{V}(H)$ as in Figure \ref{fig6}(b).
The graph in Figure \ref{fig6}(b) is an example representing
$|\overline{E}(H)| = 3$, $|[\overline{V}(H)]| = 20$ and $[\overline{V}(H)]=[4,4,3,3,3,3]$.

\begin{figure}[h]
\includegraphics{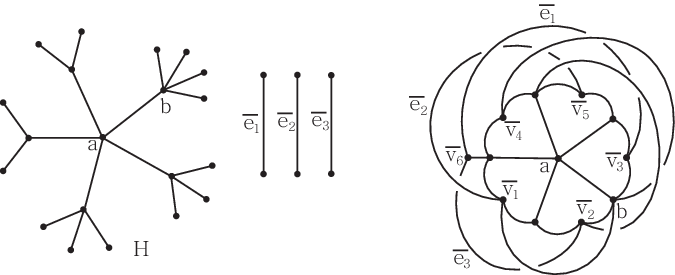}\\
\vspace{5mm}
\hspace{10mm} (a)\hspace{60mm} (b)
\caption{Visualization of $G$}
\label{fig6}
\end{figure}

\section{Case of $\deg(a) \geq 6$} \label{sec:deg6}

Recall that $G$ is a triangle-free graph with 22 edges, every vertex has degree at least 3
and $a$ has the maximal degree among all vertices.
In this section we will show that for some $b \in V$, $|\widehat{E}_{a,b}| \leq 8$.
Then as a conclusion $G$ is not intrinsically knotted by Proposition \ref{prop:planar}.
Obviously $\deg(a) < 8$ because, otherwise, $G$ would have at least 24 edges.

If $\deg(a) = 7$, then $|V_3(a)| \geq 6$ by the same reason as above.
Let $b$ be any vertex in $V(a)$.
Then $NE(a,b) \geq 9$ and $|V_3(a) \setminus \{b\}| \geq 5$.
Note that $|V_3(b)| \geq |V_3(a,b)|$.
By the count equation, $|\widehat{E}_{a,b}| \leq 8$ in $\widehat{G}_{a,b}$.

Now assume that $\deg(a) = 6$.
Since $G$ has 22 edges, $|V_3(a)| \geq 2$.
Assume that $|V_3(a)| \geq 4$.
Let $c$ be any vertex in $V_3(a)$.
Choose a vertex $b$ that has the maximal degree from $V(c) \setminus \{a\}$.
Then $|E(a)| = 6$ and $|E(b)| + |V_Y(a,b)| \geq 4$ since $|V_Y(a,b)| \geq 1$ when $\deg(b)=3$.
Thus $|\widehat{E}_{a,b}| \leq 8$.

If $|V_3(a)| = 2$ or $3$, then there must be a vertex, call it $b$, of degree 4 in $V(a)$.
If all the vertices of $V(b)$ other than $a$ have degree 3,
then $NE(a,b) = 9$ and $NV_3(a,b) \geq 5$, yielding $|\widehat{E}_{a,b}| \leq 8$.
Otherwise, there is a vertex $c$ of degree at least 4 in $V(b) \setminus \{a\}$.
Consider $\widehat{G}_{a,c}$ where $NE(a,c) = 10$.
When $|V_3(a)| = 3$, $|\widehat{E}_{a,c}| \leq 8$ because $|V_4(a,c)| \geq 1$.
When $|V_3(a)| = 2$, the graph consisting of the edges of $E(V(a))$ is $G$ itself
as in Figure \ref{fig7}.
At least one edge in $E(c)$ must be adjacent to a vertex of $V_4(a)$ other than $b$.
This implies that $|V_4(a,c)| \geq 2$ so that $|\widehat{E}_{a,c}| \leq 8$.

\begin{figure}[h]
\includegraphics{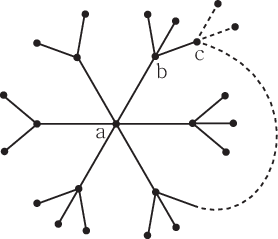}
\caption{Case of $\deg(a) = 6$}
\label{fig7}
\end{figure}

\section{Case of $\deg(a) = \deg(b) = 5$ and ${\rm dist}(a,b) \geq 2$} \label{sec:deg5dist2}

In this section we will consider the case that there are two vertices $a$ and $b$
of degree 5 with distance at least 2.
Indeed, if dist$(a,b) > 2$, $|E(V(a))| \leq 17$ because $E(V(a))$ is disjoint from $E(b)$
(similarly for $|E(V(b))|$).
This implies that $|V_3(a)| \geq 3$ and $|V_3(b)| \geq 3$.
Since they are disjoint, $|\widehat{E}_{a,b}| \leq 6$.

From now on we assume that dist$(a,b) = 2$.
Let $H$ be a subgraph of $G$ which consists of the edges of $E(V(a) \cup \{b\})$
and the  associated vertices.
Before dividing into several cases, we give some notation:
$x = V_3(a,b)$,
$y = V_4(a,b)$,
$z = V_5(a,b)$,
$u^a = V_3(a) \setminus V_3(a,b)$ and
$w^a = V(a) \setminus \{V_3(a) \cup V(b)\}$
(similarly for $b$) as indicated in Figure \ref{fig8}.
Furthermore we will write $x=\{x_1, \cdots, x_{|x|}\}$,
$u^a=\{u^a_1, \cdots, u^a_{|u^a|}\}$ and similarly for the others.
By abuse of notation, we will use the same notation
for a set and its size. For example, $x$ will also mean $|x|$.
Obviously we may assume that $u^a \leq u^b$.

\begin{figure}[h]
\includegraphics{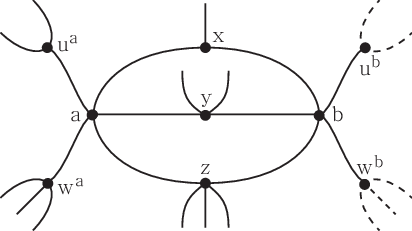}
\caption{Notations for subgraph $H$}
\label{fig8}
\end{figure}

By counting the number of edges of $H$, $10 + x + 2(y + u^a) + 3(z + w^a) \leq 22$.
Since $|V(a)| = 5$, $x + y + z + u^a + w^a = 5$.
Thus we have the following inequality;
$$y + u^a + 2(z + w^a) \leq 7.$$
If $x + y + u^a \leq 1$, then $z + w^a \geq 4$ which does not agree with the inequality above.
If $x + y + u^a \geq 4$, then $|\widehat{E}_{a,b}| \leq 8$.
Thus we consider only $x + y + u^a = 2$ or $3$.

Now we divide into six cases as follows.
If $x + y + u^a = 2$, we divide into the two cases where $x=1$ and $2$.
We can ignore the case of $x=0$ because of the inequality.
If $x + y + u^a = 3$, we have four cases: $x=0,1,2$ or $3$.
Indeed, in the case of $x + y + u^a = 3$, $|\widehat{E}_{a,b}| \leq 9$.
Furthermore we may assume that $u^b = 0$ (so $u^a = 0$) and $|V_Y(a,b)| = 0$
(i.e. $V_3(x) = \emptyset$),
otherwise $|\widehat{E}_{a,b}| \leq 8$.

\subsection{Case of $x + y + u^a = 2$ with $x = 1$}   \hspace{1cm}

Cousin 110 of the $E_9+e$ family appears in this case.
Note that $|H|=22$ and there is no extra edge.
Suppose that $z=0$ so that $w^a=3$.
If $u^a=0$, then $y=1$.
Recall that $x_1$ and $y_1$ are the unique elements of the sets $x$ and $y$ respectively.
Since $|[\overline{V}(H)]|=15$, the associated 15 edges are adjacent to vertices of $\overline{V}(H)$.
Among these 15 edges, two edges of $E(y_1)$ and three edges of $E(b)$ must be adjacent to
different vertices because $G$ is triangle-free.
So $\overline{V}(H)$ has at least 5 vertices, and we get $[\overline{V}(H)] = [3,3,3,3,3]$.
Therefore $u^b=3$ so that $|\widehat{E}_{a,b}| \leq 7$.
If $u^a=1$, then $|[\overline{V}(H)]|=16$.
Similarly among the associated 16 edges, one edge of $E(x_1)$ and four edges of $E(b)$ are adjacent to
different vertices of $\overline{V}(H)$.
So $\overline{V}(H)$ has at least 5 vertices, and we get $[\overline{V}(H)] = [4,3,3,3,3]$.
Therefore $u^b=3$ or $4$, yielding $|\widehat{E}_{a,b}| \leq 7$.

Suppose that $z=1$ so that $w^a=2$.
If $u^a=0$, then $|[\overline{V}(H)]|=14$.
Among the associated 14 edges, three edges of $E(z_1)$ and two edges of $E(b)$ are
adjacent to different vertices.
But this is impossible because each vertex of $\overline{V}(H)$ has degree at least 3.
If $u^a=1$, then $|[\overline{V}(H)]|=15$.
Among the associated 15 edges, three edges of $E(z_1)$ and three edges of $E(b)$ are adjacent to
different vertices, but this is also impossible.

Suppose that $z=2$ so that $w^a=1$.
If $u^a=0$, then $u^b+w^b=1$. Let $d$ be the unique vertex of $u^b$ or $w^b$.
Since no edge of $E(x_1)$, $E(y_1)$ or $E(z_i)$ can be adjacent to $d$,
two edges of $E(w^a)$ must simultaneously be adjacent to $d$ to make $\deg(d) \geq 3$.
If $u^a=1$, then $|[\overline{V}(H)]|=14$.
Among the  associated 14 edges, three edges of $E(z_1)$ and two edges of $E(b)$ are
adjacent to different vertices.
But this is impossible because each vertex of $\overline{V}(H)$ has degree at least 3.

Finally suppose that $z=3$ so that $w^a=0$.
Recall that $z=\{z_1, z_2, z_3\}$.
If $u^a=0$, then $u^b=w^b=0$ and $|[\overline{V}(H)]|=12$.
Thus $[\overline{V}(H)] = [5,4,3]$, $[4,4,4]$ or $[3,3,3,3]$.
When $[\overline{V}(H)] = [5,4,3]$, $G$ is homeomorphic to cousin 110 of the $E_9+e$ family
which is intrinsically knotted but not minor minimal and shown in Figure \ref{fig3}.
When $[\overline{V}(H)] = [4,4,4]$, three edges of each $E(z_i)$ are adjacent to
three different vertices of $\overline{V}(H)$.
Therefore $|V_4(z_1,z_2)|=3$ so that $|\widehat{E}_{z_1,z_2}| \leq 9$.
Since $\widehat{G}_{z_1,z_2}$ has vertex $z_3$ of degree 5,
it is not homeomorphic to $K_{3,3}$.
When $[\overline{V}(H)] = [3,3,3,3]$, three edges of $E(z_1)$ are adjacent to
three different vertices of $\overline{V}(H)$ all of which have degree 3.
Therefore $|V_3(z_1)|=3$ so that $|\widehat{E}_{a,z_1}| \leq 9$.
Since $\widehat{G}_{a,z_1}$ has vertex $b$ of degree 4,
it is not homeomorphic to $K_{3,3}$.
If $u^a=1$, then $u^b+w^b=1$.  Let $d$ be the unique vertex of $u^b$ or $w^b$.
Since no edge of $E(x_1)$ or $E(z_i)$ can be adjacent to $d$,
two edges of $E(u^a)$ are simultaneously adjacent to $d$ to make $\deg(d) \geq 3$.

\subsection{Case of $x + y + u^a = 2$ with $x = 2$}   \hspace{1cm}

In this case, $y=u^a=0$.
Suppose that $z=0$ so that $w^a=3$ and $u^b+w^b=3$.
If there is a vertex $w^a_1$ of degree 5 in the set $w^a$,
then $|H|=22$ and $|[\overline{V}(H)]|=15$.
Among these 15 edges, four edges of $E(w^a_1)$ are adjacent to different vertices of $\overline{V}(H)$.
Thus we get $[\overline{V}(H)] = [5,4,3,3]$, $[4,4,4,3]$ or $[3,3,3,3,3]$.
In any case $NV_3(b,w^a_1)+V_4(b,w^a_1) \geq 4$ so that $|\widehat{E}_{b,w^a_1}| \leq 8$.
If three vertices $w^a_1$, $w^a_2$ and $w^a_3$ have degree 4, then $|H|=21$.
Therefore there is an extra edge $\overline{e}_1$ and $|[\overline{V}(H)]|=16$.
Thus we get $[\overline{V}(H)] = [5,5,3,3]$, $[5,4,4,3]$, $[4,4,4,4]$ or $[4,3,3,3,3]$.
When $[\overline{V}(H)] = [4,3,3,3,3]$, $|V_3(b)| \geq 4$ so that $|\widehat{E}_{a,b}| \leq 8$.
For the other three cases,
$\overline{V}(H) = \{\overline{v}_1, \overline{v}_2, \overline{v}_3, \overline{v}_4\}$.
Assume that $\overline{v}_1$ and $\overline{v}_2$ are the endpoints of $\overline{e}_1$.
See Figure \ref{fig9}(a).
Since $G$ is triangle-free, we may assume that three edges of $E(b)$ are adjacent to
$\overline{v}_1$, $\overline{v}_3$ and $\overline{v}_4$.
Similarly two edges of each $E(w^a_i)$ are adjacent to $\overline{v}_3$ and $\overline{v}_4$,
and the other edge of each $E(w^a_i)$ is adjacent to $\overline{v}_1$ or $\overline{v}_2$.
Furthermore one edge of each $E(x_j)$ is adjacent to $\overline{v}_2$ to avoid a triangle.
Therefore $|V_4(b,w^a_1)| \geq 2$ so that $|\widehat{E}_{b,w^a_1}| \leq 9$.
Since $\widehat{G}_{b,w^a_1}$ has a bigon containing two vertices $a$ and $\overline{v}_2$,
it is not homeomorphic to $K_{3,3}$.

Suppose that $z=1$.
Then $|[\overline{V}(H)]|=14$ or 15 depending on the existence of an extra edge.
Thus $\overline{V}(H)$ has five vertices all of degree 3 or else fewer than five vertices.
In either case three edges of $E(z)$ and two edges of $E(b)$ must be adjacent to
different vertices of $\overline{V}(H)$.
Thus $|\widehat{E}_{b,z}| \leq 6$ or it is impossible to construct $G$.

Suppose that $z=2$ so that $w^a=1$.
If $w^a_1$ has degree 5, then $|[\overline{V}(H)]|=13$.
Since four edges of $E(w^a_1)$ are adjacent to four different vertices of $\overline{V}(H)$,
$[\overline{V}(H)] = [4,3,3,3]$.
Thus $NV_3(b,w^a_1) = 5$ so that $|\widehat{E}_{b,w^a_1}| \leq 7$.
If $w^a_1$ has degree 4, then there is an extra edge $\overline{e}_1$ and $|[\overline{V}(H)]|=14$.
Let $\overline{v}_1$ be a vertex of $\overline{V}(H)$ which is the remaining vertex of $V(b)$.
See Figure \ref{fig9}(b).
Since $G$ is triangle-free, only $\overline{e}_1$ and one edge of $E(w^a_1)$ can be
adjacent to $\overline{v}_1$.
Let $\overline{v}_2$ be the other endpoint of $\overline{e}_1$.
Thus we get $[\overline{V}(H)] = [5,3,3,3]$ or $[4,4,3,3]$.
This implies that two edges of $E(w^a_1)$ are adjacent to $\overline{v}_3$ and $\overline{v}_4$,
and three edges of each $E(z_i)$ are adjacent to $\overline{v}_2$,
$\overline{v}_3$ and $\overline{v}_4$.
Since at least two vertices of $\overline{V}(H)$ have degree 3,
$NV_3(b,z_1) \geq 4$ so that $|\widehat{E}_{b,z_1}| \leq 9$.
Since $\widehat{G}_{b,z_1}$ has vertex $a$ with degree 4,
it is not homeomorphic to $K_{3,3}$.

\begin{figure}[h]
\includegraphics{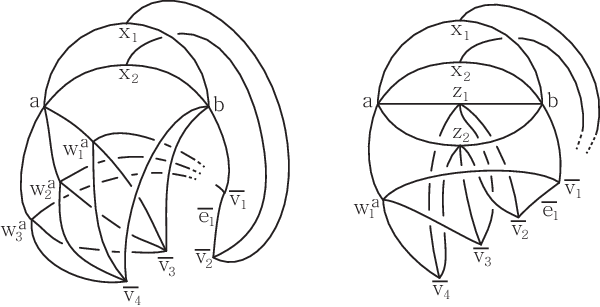}\\
\vspace{5mm}
\hspace{6mm} (a)\hspace{50mm} (b)
\caption{Case of $x + y + u^a = 2$ with $x = 2$}
\label{fig9}
\end{figure}

Suppose that $z=3$.
Then $|H|=21$ so that there is an extra edge $\overline{e}_1$ and $|[\overline{V}(H)]|=13$.
Since two edges of $E(z_1)$ cannot simultaneously be adjacent to the endpoints of $\overline{e}_1$,
$\overline{V}(H)$ needs four vertices and so $[\overline{V}(H)] = [4,3,3,3]$.
Since $NV_3(a,z_1) \geq 4$, $|\widehat{E}_{a,z_1}| \leq 9$.
Since $\widehat{G}_{a,z_1}$ has vertex $b$ with degree 4,
it is not homeomorphic to $K_{3,3}$.

\subsection{Case of $x + y + u^a = 3$ with $x = 0$}   \hspace{1cm}

Recall that $x=0$, $y=3$, $u^a=u^b=0$ (as mentioned before Case 4.1.) and $|\widehat{E}_{a,b}| \leq 9$.
Suppose that $z=0$ or 1.
Note that $w^a = w^b = 2 - z$.
Since $|H|=22$ and $\deg(w^b_1) \geq 4$, at least three edges of $E(w^b_1)$ must be adjacent to
different vertices of $w^a$, a contradiction to $w^a \leq 2$.

Suppose that $z=2$.
Since $|[\overline{V}(H)]|=12$, we get $[\overline{V}(H)] = [5,4,3]$, $[4,4,4]$ or $[3,3,3,3]$.
When $[\overline{V}(H)] = [5,4,3]$ or $[4,4,4]$,
$|\widehat{E}_{a,b}| \leq 9$ and $\widehat{G}_{a,b}$ has vertices with degree 4 or 5.
Consider $[\overline{V}(H)] = [3,3,3,3]$.
If $V(z_1) \neq V(z_2)$, then $NV_3(z_1,z_2)=4$ so that $|\widehat{E}_{z_1,z_2}| \leq 8$.
Assume that $V(z_1) = V(z_2) = \{a,b,\overline{v}_1,\overline{v}_2,\overline{v}_3\}$.
So $V(\overline{v}_4) = \{y_1,y_2,y_3\}$ and we have the graph of Figure \ref{fig10}(a).
Then $|\widehat{E}_{z_1,y_1}| \leq 9$
and $\widehat{G}_{z_1,y_1}$ has vertex $z_2$ with degree 4.

\begin{figure}[h]
\includegraphics{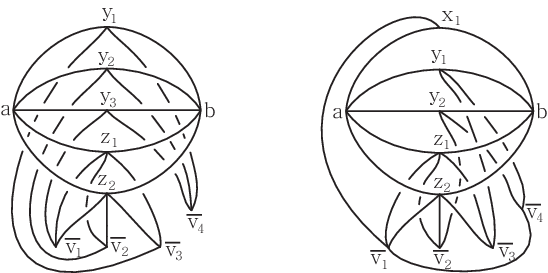}\\
\vspace{5mm}
(a)\hspace{52mm} (b)
\caption{Cases of $x + y + u^a = 3$ with $x = 0$ or $x = 1$}
\label{fig10}
\end{figure}

\subsection{Case of $x + y + u^a = 3$ with $x = 1$}   \hspace{1cm}

The new graph $M_{11}$ appears in this case.
Note that $x=1$, $y=2$, $u^a=u^b=0$, $|\widehat{E}_{a,b}| \leq 9$
and possibly there is an extra edge $\overline{e}_1$.
Suppose that $z=0$.
Note that $w^a=w^b=2$.
Then two edges of $E(w^b_1)$ must be adjacent to different vertices of $w^a$
and another edge of $E(w^b_1)$ is $\overline{e}_1$.
The same thing happens for $E(w^b_2)$ so that $G$ has the triangle
consisting of $\overline{e}_1$ and two edges of $E(b)$.
Suppose that $z=1$.
Then, even though one edge of $E(w^b_1)$ is $\overline{e}_1$,
two edges of $E(w^b_1)$ must be adjacent to $w^a_1$, a contradiction.

Suppose that $z=2$.
Since $|[\overline{V}(H)]|=13$, we get $[\overline{V}(H)] = [5,5,3]$, $[5,4,4]$ or $[4,3,3,3]$.
When $[\overline{V}(H)] = [5,5,3]$ or $[5,4,4]$,
$|\widehat{E}_{a,b}| \leq 9$ and $\widehat{G}_{a,b}$ has vertices with degree 4 or 5.
Suppose $[\overline{V}(H)] = [4,3,3,3]$.
If three vertices of $V(z_i)$ other than $a$ and $b$ have all degree 3,
then $|\widehat{E}_{a,z_1}| \leq 9$
and $\widehat{G}_{a,z_1}$ has vertex $b$ with degree 4, a contradiction.
Thus one edge of each $E(z_i)$ is adjacent to $\overline{v}_1$ of degree 4.
If $V(z_1) \neq V(z_2)$, then $NV_3(z_1,z_2) = 3$ and $V_4(z_1,z_2)$ contains $\overline{v}_1$,
yielding $|\widehat{E}_{z_1,z_2}| \leq 8$.
Assume that $V(z_1) = V(z_2) = \{a,b,\overline{v}_1,\overline{v}_2,\overline{v}_3\}$.
Since $V_Y(a,b)$ is empty, one edge of $E(x_1)$ is adjacent to $\overline{v}_1$.
Now the third edge of each $E(\overline{v}_i)$, $i=2,3$, is adjacent to a vertex of set $y$,
otherwise $|\widehat{E}_{z_1,z_2}| \leq 9$
and $\widehat{G}_{z_1,z_2}$ has a vertex $y_1$ or $y_2$ with degree 4.
Thus we have the graph of Figure \ref{fig10}(b).
This graph $G$ is indeed $M_{11}$.

\subsection{Case of $x + y + u^a = 3$ with $x = 2$}   \hspace{1cm}

Again the graph $M_{11}$ appears in this case.
Note that $x=2$, $y=1$, $u^a=u^b=0$, $|\widehat{E}_{a,b}| \leq 9$
and possibly there are two extra edges $\overline{e}_1$ and $\overline{e}_2$.
Suppose that $z=0$.
Obviously extra edges and edges of only $E(w^a_1)$ and $E(w^a_2)$
can be adjacent to $w^b_1$ and $w^b_2$.
Since an extra edge cannot be adjacent to both $w^b_1$ and $w^b_2$ simultaneously,
there are two edges of $E(w^a_1)$ adjacent to $w^b_1$ and $w^b_2$ respectively.
Therefore $|V_4(b,w^a_1)|=2$ and $|V_3(b)|=2$, yielding $|\widehat{E}_{b,w^a_1}| \leq 9$.
Since $\widehat{G}_{b,w^a_1}$ has vertex $a$ with degree 4,
it is not homeomorphic to $K_{3,3}$.

Suppose that $z=1$.
Then two extra edges and one edge of $E(w^a_1)$ must be adjacent to $w^b_1$.
See Figure \ref{fig11}(a).
Since $w^b_1$ has degree 4 and $|[\overline{V}(H)]|=15$,
we get $[\overline{V}(H)] = [5,4,3,3]$ or $[4,4,4,3]$.
This implies that $\overline{V}(H)$ has four vertices including $w^b_1$.
Then the remaining two edges of $E(w^a_1)$ and two extra edges must be adjacent to
three vertices in $\overline{V}(H)$ other than $w^b_1$.
Thus $G$ has a triangle consisting of two edges of $E(w^a_1)$ and an extra edge.

\begin{figure}[h]
\includegraphics{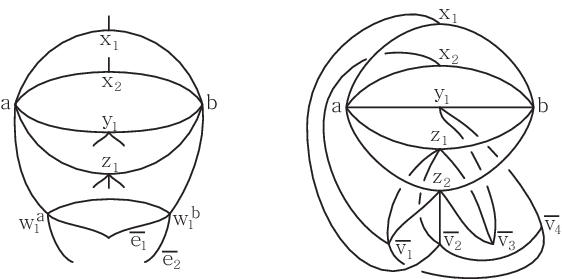}\\
\vspace{5mm}
(a)\hspace{50mm} (b)
\caption{Case of $x + y + u^a = 3$ with $x = 2$}
\label{fig11}
\end{figure}

Suppose that $z=2$.
Since $|[\overline{V}(H)]|=14$, we get $[\overline{V}(H)] = [5,5,4]$, $[5,3,3,3]$ or $[4,4,3,3]$.
When $[\overline{V}(H)] = [5,5,4]$,
$|\widehat{E}_{a,b}| \leq 9$ and $\widehat{G}_{a,b}$ has vertices
with degree 4 or 5 in $\overline{V}(H)$.
When $[\overline{V}(H)] = [5,3,3,3]$, $|V_3(z_1)| \geq 2$.
Therefore $|\widehat{E}_{b,z_1}| \leq 9$ and $\widehat{G}_{b,z_1}$ has vertex $a$ with degree 4.
When $[\overline{V}(H)] = [4,4,3,3]$,
we assume that $|V_3(z_1)|=|V_3(z_2)|=1$ by the same reason as before.
Recall that $\overline{v}_1$ and $\overline{v}_2$ have degree 4,
and $\overline{v}_3$ and $\overline{v}_4$ have degree 3.
See Figure \ref{fig11}(b).
Two edges of each $E(z_i)$ must be adjacent to $\overline{v}_1$ and $\overline{v}_2$ respectively.
The remaining two edges of $E(z_1)$ and $E(z_2)$ are adjacent to one vertex, say $\overline{v}_3$,
otherwise $NV_3(z_1,z_2)=2$ and $|V_4(z_1,z_2)|=2$, yielding $|\widehat{E}_{z_1,z_2}| \leq 8$.
Furthermore $V_Y(z_1,z_2)$ must be empty,
so the remaining edge of $E(\overline{v}_3)$ is adjacent to $y_1$.
Also $V_Y(a,b)$ must be empty,
so one edge of each $E(x_j)$ is adjacent to $\overline{v}_1$ or $\overline{v}_2$.
Combining all these conditions, we get the graph $M_{11}$
which is the same graph as in Figure \ref{fig10}(b)
with $\{a,b\}$ and $\{z_1,z_2\}$ exchanged.

\subsection{Case of $x + y+ u^a = 3$ with $x = 3$}   \hspace{1cm}

In this case $x = 3$, $y=u^a=u^b=0$, $|\widehat{E}_{a,b}| \leq 9$
and possibly there are three extra edges $\overline{e}_1$, $\overline{e}_2$ and $\overline{e}_3$.
First suppose that all three edges each of which is from $E(x_i)$, $i=1,2,3$,
are adjacent to one vertex $c$ which is neither $a$ nor  $b$.
This means that $V(c)$ contains $V_3(a,b)$.
If $\deg(c) \geq 4$, then $|\widehat{E}_{a,b}| \leq 8$.
If $\deg(c) = 3$, then $G$ is the union of two components which are $K_{3,3}$ and $G \setminus K_{3,3}$.
Since neither component is intrinsically knotted and they meet at only two vertices $a$ and $b$,
their union $G$ cannot be intrinsically knotted.

Now assume that the three edges mentioned above are not all adjacent to one vertex.
Recall that $|V_Y(a,b)|=0$ as mentioned before Case 4.1.
Suppose that $z=0$.
If one of $w^a$, say $w^a_1$, has degree 5,
then  $|\widehat{E}_{b,w^a_1}| \leq 9$ and
$\widehat{G}_{b,w^a_1}$ has vertex $a$ with degree 4.
Thus $w^a_1$ and $w^a_2$ have degree 4.
Since $|[\overline{V}(H)]|=17$,
we get $[\overline{V}(H)] = [5,5,4,3]$, $[5,4,4,4]$, $[5,3,3,3,3]$ or $[4,4,3,3,3]$.
Since $w^b = 2$ and $|V_Y(a,b)|=0$, $\overline{V}(H)$ has four vertices with degree more than 3.
Thus we consider the only case $[5,4,4,4]$ which induces the graph of Figure \ref{fig12}(a).
Note that at least one edge of $E(w^a_1)$ is adjacent to a vertex of $w^b$.
Therefore $|V_4(b,w^a_1)| \geq 1$ and $|V_3(b)|=3$, yielding $|\widehat{E}_{b,w^a_1}| \leq 9$.
Since $\widehat{G}_{b,w^a_1}$ has vertex $w^a_2$ with degree 4,
it is not homeomorphic to $K_{3,3}$.

Suppose that $z=1$.
Since $|[\overline{V}(H)]|=16$, by reasoning similar to that above,
we need only consider the case $[\overline{V}(H)] = [4,4,4,4]$.
Take a vertex $\overline{v}_1$ with degree 4 from $V(w^a_1)$.
Then $V_4(a,\overline{v}_1) = \{w^a_1\}$ and $|V_3(a)|=3$,
yielding $|\widehat{E}_{a,\overline{v}_1}| \leq 9$ as in Figure \ref{fig12}(b).
Since $\widehat{G}_{a,\overline{v}_1}$ has vertex $b$ with degree 4,
it is not homeomorphic to $K_{3,3}$.

Suppose that $z=2$.
Since $|[\overline{V}(H)]|=15$,
we get $[\overline{V}(H)] = [5,5,5]$, $[5,4,3,3]$, $[4,4,4,3]$ or $[3,3,3,3,3]$.
Since $|V_Y(a,b)| = 0$, $[\overline{V}(H)] \neq [3,3,3,3,3]$.
The existence of three extra edges implies $[\overline{V}(H)] \neq [5,5,5]$ because $G$ is triangle-free.
Thus $[\overline{V}(H)] = [5,4,3,3]$ or $[4,4,4,3]$, so that $\overline{V}(H)$ consists of four vertices.
Since three edges of each $E(z_i)$ are adjacent to different vertices of $\overline{V}(H)$
and $G$ is triangle-free, we have $V(z_1)=V(z_2)$ producing the graph of Figure \ref{fig12}(c).
Let $\overline{v}_1$ be the vertex that the three extra edges are adjacent to.
Note that at least one of $\overline{v}_i$, $i=2,3,4$, must have degree greater than 3.
Now we consider $G$ where two sets of vertices $\{a,b\}$ and $\{z_1,z_2\}$ are exchanged.
Then in $G$ with the new labelling, $x < 3$.
But we've already completed the argument for all cases except the case $x=3$ and $z=2$.

\begin{figure}[h]
\includegraphics{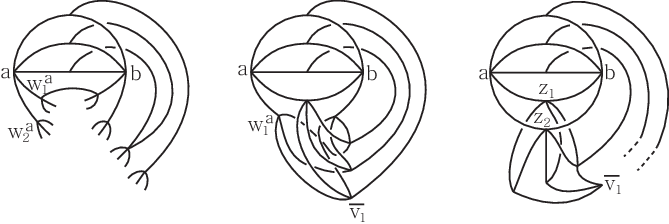}\\
\vspace{3mm}
(a) \hspace{35mm} (b) \hspace{35mm} (c)
\caption{Case of $x + y+ u^a = 3$ with $x = 3$}
\label{fig12}
\end{figure}

\section{Case of $\deg(a) = \deg(b) = 5$ and ${\rm dist}(a,b) = 1$} \label{sec:deg5dist1}

In this section we will consider the case that there are exactly two vertices $a$ and $b$
of degree 5 and their distance is 1.
So all the other vertices have degree 3 or 4.
We may assume that $|V_3(a)| \geq |V_3(b)|$.
Let $H$ be a subgraph of $G$ which consists of the edges of $E(V(a) \cup \{b\})$
and the  associated vertices.
Now we divide into five cases according to the number of vertices of $V_3(a)$.

\subsection{Case of $|V_3(a)|=0$}   \hspace{1cm}

In this case all eight vertices of $V(a)$ and $V(b)$ except $a$ and $b$ have degree 4.
Since $|H|=21$ and there is only one extra edge, $|[\overline{V}(H)]|=18$.
So we get $[\overline{V}(H)] = [4,4,4,3,3]$ or $[3,3,3,3,3,3]$.
Both cases are impossible because the set $V(b)$ has four vertices of degree 4 in $\overline{V}(H)$.

\subsection{Case of $|V_3(a)|=1$}   \hspace{1cm}

Let $c$ be the unique vertex of $V_3(a)$.
Since $V(a)$ has three vertices of degree 4,
$|H|=20$ and there are two extra edges $\overline{e}_1$ and $\overline{e}_2$.
Therefore $|[\overline{V}(H)]|=19$ and
$[\overline{V}(H)] = [4,4,4,4,3]$ or $[4,3,3,3,3,3]$.
The latter case is impossible because $V(b)$ has three or four vertices of degree 4.
In the former case, only $\overline{v}_5$ has degree 3.
First assume that $V(b)$ has four vertices of degree 4 and
so $\overline{v}_5$ is not a vertex of $V(b)$.
We can choose a vertex $d$ from $V(b)$
such that none of edges of $E(d)$ is adjacent to vertex $c$.
This implies that three edges of $E(d)$ must be adjacent to vertices of $V(a)$ with degree 4
or $\overline{v}_5$.
Then $NE(a,d) = 9$ and $NV_3(a,d) + |V_4(a,d)| = 4$, yielding $|\widehat{E}_{a,d}| \leq 9$.
Since $\widehat{G}_{a,d}$ still has a vertex of degree 4 in $V(b)$,
it is not homeomorphic to $K_{3,3}$.

Now we may say that $V(b)$ has three vertices of degree 4 and a vertex $d$ of degree 3
which is indeed $\overline{v}_5$.
Let $\overline{v}_1$ be the unique vertex of degree 4 in $\overline{V}(H)$ excluding $V(b)$.
Now we divide into three cases.
If $V(\overline{v}_1)$ contains $c$ and $d$ as in Figure \ref{fig13}(a),
then $V(b)$ contains a vertex $e$ such that three edges of $E(e)$ are adjacent to $V_4(a)$.
Therefore $|\widehat{E}_{a,e}| \leq 9$,
but $\widehat{G}_{a,e}$ still has vertex $\overline{v}_1$ of degree 4.
If $V(\overline{v}_1)$ contains only $c$ (or similarly only $d$) as in Figure \ref{fig13}(b),
then $NV_3(b,\overline{v}_1) + |V_4(b,\overline{v}_1)| = 4$
so that $|\widehat{E}_{b,\overline{v}_1}| \leq 9$.
But $\widehat{G}_{b,\overline{v}_1}$ still has a vertex of degree 4 in $V(a)$.
Finally assume that $V(\overline{v}_1)$ does not contain $c$ and $d$ as in Figure \ref{fig13}(c).
Let $c_1$ and $d_1$ be the vertices of $V_4(a)$ and $V_4(b)$ excluding $V(\overline{v}_1)$ respectively.
$V(c_1)$ must contains $d$; otherwise, $V(c_1)$ contains the three vertices of $V_4(b)$,
yielding $NV_3(b,c_1) + |V_4(b,c_1)| = 4$ so that $|\widehat{E}_{b,c_1}| \leq 9$.
But $\widehat{G}_{b,c_1}$ still has vertex $\overline{v}_1$ of degree 4.
Let $c_2$ and $c_3$ be the other vertices of $V_4(a)$.
Since $G$ is a triangle-free graph, one edge of $E(c_2)$ and one edge of $E(c_3)$
are adjacent to $d$, contradicting the fact that $d$ has degree 3.

\begin{figure}[h]
\includegraphics{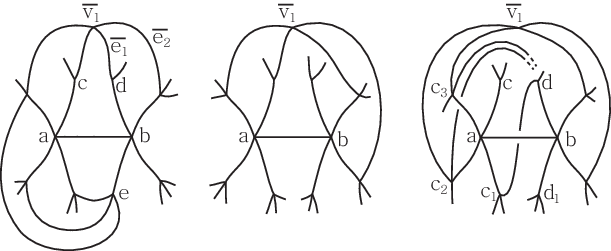}\\
\vspace{3mm}
(a) \hspace{27mm} (b) \hspace{28mm} (c)
\caption{Case of $|V_3(a)| = |V_3(b)| =1$}
\label{fig13}
\end{figure}

\subsection{Case of $|V_3(a)|=2$}   \hspace{1cm}

Cousin 94 of the $E_9+e$ family appears in this case.
Here $|H|=19$ and there are three extra edges.
Therefore $|[\overline{V}(H)]|=20$ and
$[\overline{V}(H)] = [4,4,4,4,4]$ or $[4,4,3,3,3,3]$.
In the former case, four vertices of $V(b)$ have degree 4.
Let $\overline{v}_5$ be the last vertex of $\overline{V}(H)$ excluding $V(b)$.
Then three extra edges must connect $\overline{v}_5$
with three different vertices, say $\overline{v}_1$, $\overline{v}_2$ and $\overline{v}_3$,
in $V(b)$ respectively.
Let $l$ be the non-extra edge of $E(\overline{v}_5)$ and
$c$ be the other vertex that $l$ is adjacent to.
If $c$ has degree 3,
then $NV_3(b,\overline{v}_5) + |V_4(b,\overline{v}_5)| = 4$
so that $|\widehat{E}_{b,\overline{v}_5}| \leq 9$.
But $\widehat{G}_{b,\overline{v}_5}$ still has vertex $a$ of degree 4.
If $c$ has degree 4,
then the remaining two edges of $E(c)$ must be adjacent to four vertices of $V(b)$.
But they cannot be adjacent to $\overline{v}_1$, $\overline{v}_2$ and $\overline{v}_3$
to avoid producing triangles in $G$.
So they are indeed adjacent to the same vertex $\overline{v}_4$ which is the remaining vertex in $V(b)$,
producing a bigon which is not allowed in $G$.

For the latter case $[4,4,3,3,3,3]$, obviously $|V_3(b)|=2$.
Let $c_1$, $c_2$, $c_3$ and $c_4$ be two vertices of degree 3
and the other two vertices of degree 4 of $V(a)$.
And similarly we define $d_1$, $d_2$, $d_3$ and $d_4$,
and let $\overline{v}_5$ and $\overline{v}_6$ denote the remaining vertices with degree 3 of
$\overline{V}(H)$ excluding $V(b)$.
See Figure \ref{fig14}(a).
First we will show that $V(c_1)$ (or $V(c_2)$) contains both $d_3$ and $d_4$,
and by the same reason $V(d_1)$ (or $V(d_2)$) contains both $c_3$ and $c_4$.
If $V(d_3)$ contains neither $c_1$ nor $c_2$,
then it must contain three vertices among $c_3$, $c_4$, $\overline{v}_5$ and $\overline{v}_6$.
Thus $NV_3(a,d_3) + |V_4(a,d_3)| \geq 5$ so that $|\widehat{E}_{a,d_3}| \leq 8$.
And if $V(d_3)$ contains only $c_1$ (or only $c_2$),
then similarly $NV_3(a,d_3) + |V_4(a,d_3)| \geq 4$ so that $|\widehat{E}_{a,d_3}| \leq 9$.
But $\widehat{G}_{a,d_3}$ still has vertex $d_4$ of degree 4 except
when there is an edge connecting $c_1$ and $d_4$ as we expect.
If $V(d_3)$ contains both $c_1$ and $c_2$,
still $V(d_4)$ contains $c_1$ or $c_2$ or both.
This implies that $V(c_1)$ (or $V(c_2)$) contains both $d_3$ and $d_4$.
Now consider the graph $\widehat{G}_{a,b}$.
Since $|\widehat{E}_{a,b}| \leq 9$,
the only exceptional case we handle is when it is homeomorphic to $K_{3,3}$.
In $\widehat{G}_{a,b}$,
dist$(c_3,c_4) =1$ and dist$(d_3,d_4) =1$, so that dist$(\overline{v}_5,\overline{v}_6) =1$
because of the property of the graph $K_{3,3}$.
Without loss of generality we may assume that
$\widehat{G}_{a,b}$ is the graph in Figure \ref{fig14}(b).

\begin{figure}[h]
\includegraphics{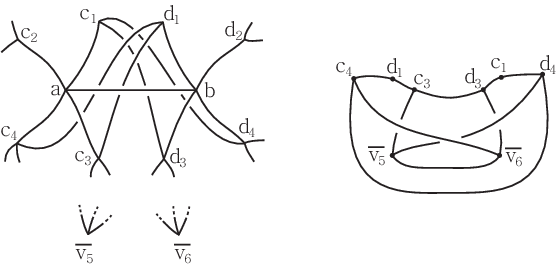}\\
\vspace{3mm}
\hspace{4mm} (a) \hspace{46mm} (b)
\caption{Case of $|V_3(a)| = |V_3(b)| =2$}
\label{fig14}
\end{figure}

Now we try to find proper locations for the two vertices $c_2$ and $d_2$ in $\widehat{G}_{a,b}$
so that we can figure out the original graph $G$ by adding $E(\{a,b\})$.
If $c_2$ lies on the edge connecting $c_3$ and $d_3$ (or similarly $c_4$ and $d_4$),
then $d_2$ must lie between $c_2$ and $c_3$ on the same edge preventing a triangle in $G$.
In this case $NV_3(c_3,d_4) = 4$ and $|V_Y(c_3,d_4)| = 1$, so that $|\widehat{E}_{c_3,d_4}| \leq 9$.
But $\widehat{G}_{c_3,d_4}$ contains vertex $b$ of degree 4.
If $c_2$ lies on the edge connecting $d_3$ and $\overline{v}_6$
(or similarly $d_4$ and $\overline{v}_5$),
then possible locations for $d_2$ are at the four dots shown in Figure \ref{fig15}(a).
For all four cases, $|\widehat{E}_{b,c_3}| \leq 9$
but $\widehat{G}_{b,c_3}$ still contains a triangle whose vertices are $a$, $c_1$ and $c_2$.
If $c_2$ lies on the edge connecting $c_3$ and $\overline{v}_5$
(or similarly $c_4$ and $\overline{v}_6$),
then $d_2$ must lie on the edge connecting $c_2$ and $c_3$.
Then similarly $|\widehat{E}_{a,d_3}| \leq 9$
but $\widehat{G}_{a,d_3}$ contains a triangle whose vertices are $a$, $d_1$ and $d_2$.
Finally we only need to consider the remaining case
that both $c_2$ and $d_2$ lie on the edge connecting $\overline{v}_5$ and $\overline{v}_6$
as in Figure \ref{fig15}(b), although
there is the possibility of switching their locations.
Cousin 94 of the $E_9+e$ family appears in either case.

\begin{figure}[h]
\includegraphics{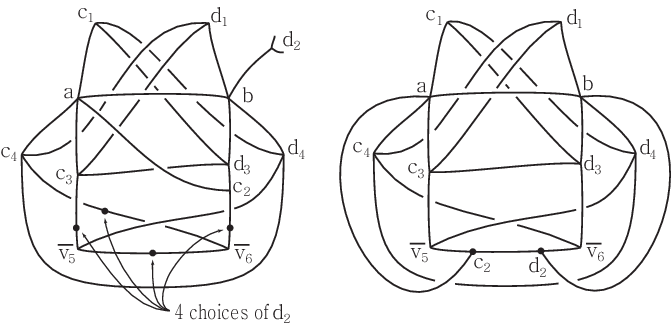}\\
\vspace{3mm}
(a) \hspace{52mm} (b)
\caption{Cousin 94 of the $E_9+e$ family}
\label{fig15}
\end{figure}

\subsection{Case of $|V_3(a)|=3$}   \hspace{1cm}

In this case $V(b)$ has at most one vertex of degree 3;
otherwise, $NE(a,b) = 9$ and $NV_3(a,b) \geq 5$, yielding $|\widehat{E}_{a,b}| \leq 8$.
Since $|H|=18$ and there are four extra edges, $|[\overline{V}(H)]|=21$
and $[\overline{V}(H)] = [4,4,4,3,3,3]$ or $[3,3,3,3,3,3,3]$.
Here the latter case is impossible because $V_3(b)$ has at most one vertex.
In the former case $V(b)$ contains three vertices of degree 4 and a vertex of degree 3.
Let $\overline{v}_5$ and $\overline{v}_6$ be the vertices of degree 3
in $\overline{V}(H)$ excluding $V(b)$.
There are at least two edges among the four extra edges
each of which connects a vertex of $V_4(b)$ and one of $\overline{v}_5$ or $\overline{v}_6$.
Let $c_1$ and $c_2$ be the vertices of $V_4(b)$  associated to two extra edges.
If $c_1 = c_2$, then $NV_3(a,c_1) \geq 5$ so that $|\widehat{E}_{a,c_1}| \leq 8$.
Thus $c_1$ and $c_2$ are different and let $c_3$ be the other vertex of $V_4(b)$.
In this case $NV_3(a,c_1) \geq 4$ so that $|\widehat{E}_{a,c_1}| \leq 9$.
Except for the case in Figure \ref{fig16},
$\widehat{G}_{a,c_1}$ has a vertex of degree 4 which is $c_2$ or $c_3$.
For the exceptional case,
similarly $|\widehat{E}_{a,c_2}| \leq 9$
but $\widehat{G}_{a,c_2}$ contains a triangle whose vertices are $b$, $c_1$ and $c_3$.
Thus it is not homeomorphic to $K_{3,3}$.

\begin{figure}[h]
\includegraphics{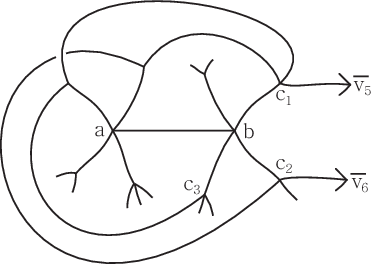}\\
\caption{Case of $|V_3(a)| = 3$ and $|V_3(b)| =1$}
\label{fig16}
\end{figure}

\subsection{Case of $|V_3(a)|=4$}   \hspace{1cm}

In this case $V_3(b)$ must be the empty set;
otherwise, $NV_3(a,b) \geq 5$ so that $|\widehat{E}_{a,b}| \leq 8$.
Since $|H|=17$ and there are five extra edges, $|[\overline{V}(H)]|=22$
and $[\overline{V}(H)] = [4,4,4,4,3,3]$ or $[4,3,3,3,3,3,3]$.
The latter case is impossible because the set $V_3(b)$ is empty.
In the former case all four vertices of $V(b)$ except $a$ have degree 4.
Let $\overline{v}_5$ and $\overline{v}_6$ be the vertices of degree 3 in $\overline{V}(H)$
which are obviously not in $V(b)$.
There is a vertex $c$ in $V(b)$ so that $E(c)$ has an extra edge connecting $c$ and
one of $\overline{v}_5$ or $\overline{v}_6$.
Then $NV_3(a,c) \geq 5$ so that $|\widehat{E}_{a,c}| \leq 8$.


\begin{thebibliography}{AA}
\bibitem{BM} J. Barsotti and T. W. Mattman,
    {\em Intrinsically knotted graphs with 21 edges},
    arXiv:1303.6911.
\bibitem{BBFFHL} P. Blain, G. Bowlin, T. Fleming, J. Foisy, J. Hendricks, and J. LaCombe,
    {\em Some results on intrinsically knotted graphs},
    J. Knot Theory Ramifications \textbf{16} (2007) 749--760.
\bibitem{CG} J. Conway and C. McA. Gordon,
    {\em Knots and links in spatial graphs},
    J. Graph Theory \textbf{7} (1985) 445--453.
\bibitem{FN} E. Flapan and R. Naimi,
    {\em The $Y \nabla$ moves does not preserve intrinsic knottedness},
    Osaka J. Math. \textbf{45} (2008) 107--111.
\bibitem{F} J. Foisy,
    {\em Intrinsically knotted graphs},
    J. Graph Theory \textbf{39} (2002) 178--187.
\bibitem{GMN} N. Goldberg, T. Mattman, and R. Naimi,
    {\em Many, many more intrinsically knotted graphs},
    Algebr. Geom. Topol. \textbf{14} (2014) 1801--1823.
\bibitem{HNTY} R. Hanaki, R. Nikkuni, K. Taniyama, and A. Yamazaki,
    {\em On intrinsically knotted or completely $3$-linked graphs},
    Pacific J. Math. \textbf{252} (2011) 407--425.
\bibitem{JKM} B. Johnson, M. Kidwell, and T. Michael,
    {\em Intrinsically knotted graphs have at least $21$ edges},
    J. Knot Theory Ramifications \textbf{19} (2010) 1423--1429.
\bibitem{KMO} H. Kim, T. Mattman, and S. Oh,
    {\em Bipartite intrinsically knotted graphs with 22 edges},
    (preprint).
\bibitem{KS} T. Kohara and S. Suzuki,
    {\em Some remarks on knots and links in spaital graphs},
    Knots 90 (Osaka, 1990) (1992) 435--445.
\bibitem{LKLO} M. Lee, H. Kim, H. J. Lee, and S. Oh,
    {\em Exactly fourteen intrinsically knotted graphs have 21 edges},
    Algebr. Geom. Topol. (in press).
\bibitem{MRS} R. Motwani, A. Raghunathan, and H. Saran,
    {\em Constructive results from graph minors; linkless embeddings},
    Proc. 29th Annual Symposium on Foundations of Computer Science,
    IEEE (1988) 398--409.
\bibitem{OT} M. Ozawa and Y. Tsutsumi,
    {\em Primitive spatial graphs and graph minors},
    Rev. Mat. Complut. \textbf{20} (2007) 391--406.
\bibitem{RS} N. Robertson and P. Seymour,
    {\em Graph minors XX, Wagner's conjecture},
    J. Combin. Theory Ser. B \textbf{92} (2004) 325--357.
\bibitem{S} H. Sachs,
    {\em On spatial representations of finite graphs},
    Colloq. Math. Soc. Janos Bolyai \textbf{37} (1984) 649--662.
\end{thebibliography}
\end{document}